\documentclass{amsart}[12pt]
\usepackage{graphicx,graphics, amsmath, amsthm, amscd, amsfonts}

\usepackage{latexsym}
\makeatletter \oddsidemargin.9375in \evensidemargin \oddsidemargin
\newtheorem{theorem}{Theorem}[section]
\newtheorem{lemma}[theorem]{Lemma}
\newtheorem{proposition}[theorem]{Proposition}

\theoremstyle{definition}

\numberwithin{equation}{section}

\begin{document}
\Large
\title[ ]
{Additive maps preserving idempotency of products
or Jordan products of operators}

\author[Ali Taghavi and Roja Hosseinzadeh]{Ali Taghavi and Roja
Hosseinzadeh}

\address{{ Department of Mathematics, Faculty of Mathematical Sciences,
 University of Mazandaran, P. O. Box 47416-1468, Babolsar, Iran.}}

\email{taghavi@umz.ac.ir,  ro.hosseinzadeh@umz.ac.ir}



\begin{abstract}\large
Let $\mathcal{H}$ and $\mathcal{K}$ be infinite dimensional Hilbert spaces, while $\mathcal{B(H)}$ and
$\mathcal{B(K)}$ denote the algebras of all linear bounded operators on $\mathcal{H}$ and $\mathcal{K}$, respectively. We characterize the forms of additive mappings from $\mathcal{B(H)}$ into $\mathcal{B(K)}$
that preserve the nonzero idempotency of either Jordan products of operators or usual products of operators in both directions.
\end{abstract} \maketitle

\section{Introduction And Statement of the Results}
\noindent  The study of maps on operator algebras
preserving certain properties or subsets is a topic which
attracts much attention of many authors. Some of the problems
are concerned with preserving a certain property of usual product
or other products of operators. For example see $[3,5-9,11,13,14]$.\par
Let $\mathcal{R}$ and $ \mathcal{R'}$ be two rings and $\phi
:\mathcal{R} \rightarrow \mathcal{R'}$ be a map. Denote by $P_
\mathcal{R}$ and $P_ \mathcal{R'}$ the set of all idempotent
elements of $ \mathcal{R}$ and $ \mathcal{R'}$, respectively. We
say that $\phi$ preserves the idempotency of product of two elements, the
idempotency of triple Jordan product of two elements and the idempotency of Jordan product of two elements, whenever we have
$$AB \in P_ \mathcal{R} \Rightarrow \phi (A) \phi(B) \in P_ \mathcal{R'},$$
$$ ABA \in P_ \mathcal{R} \Rightarrow \phi (A) \phi(B) \phi(A) \in P_ \mathcal{R'}$$
and
$$ \frac{1}{2}( AB+BA) \in P_ \mathcal{R} \Rightarrow \frac{1}{2}( \phi (A) \phi(B)+ \phi (B) \phi(A) )\in P_ \mathcal{R'},$$
respectively.
The triple Jordan product and the Jordan product of two elements $A$ and $B$ are defined as $ABA$ and $ \frac{1}{2}( AB+BA) $, respectively.
Let $\mathcal{H}$ and $\mathcal{K}$ be infinite dimensional Hilbert spaces, while $\mathcal{B(H)}$ and $\mathcal{B(K)}$ denote the algebras of all linear bounded operators on $\mathcal{H}$ and $\mathcal{K}$, respectively. In $[7]$, authors characterized  the forms of unital surjective maps on $B(X)$ preserving the nonzero idempotency of product of operators in both directions. Also in $[13]$, authors characterized the forms of linear surjective maps on $B(X)$ preserving the nonzero
idempotency of either products of operators or triple Jordan products of operators. \par
In this paper, we determine the form of additive mapping $\phi:\mathcal{B(H)}\rightarrow  \mathcal{B(K)}$ such that the range of $\phi$ contains all minimal idempotents and $I$ and also $\phi$ preserves the nonzero idempotency of Jordan products of operators in both directions. Moreover, we determine the form of surjective additive mapping $\phi:\mathcal{B(H)}\rightarrow  \mathcal{B(K)}$ that preserves the nonzero idempotency of usual products of operators in both directions. Our main result are as follows.
%---------------------------------------------------------------------------------------%
\begin{theorem} Let $\mathcal{H}$ and $\mathcal{K}$ be two infinite dimensional real or complex Hilbert spaces and $\phi:\mathcal{B(H)}\rightarrow  \mathcal{B(K)}$ be an additive map such that the range of $\phi$ contains all minimal idempotents and $I$. If $\phi$  preserves the nonzero idempotency of Jordan products of operators in both directions, then $\phi$ either annihilates minimal idempotents or there exists a bounded linear or conjugate linear bijection $A: \mathcal{H} \rightarrow \mathcal{K}$ such that $ \phi(T)= \xi ATA^{-1}$ for every $T \in \mathcal{B(H)}$ or $ \phi(T)= \xi AT^tA^{-1}$ for every $T \in \mathcal{B(H)}$, where $ \xi = \pm 1$ ( in the case that $\mathcal{H}$ and $\mathcal{K}$ are real, $A$ is linear).
\end{theorem}
%---------------------------------------------------------------------------------------%
\begin{theorem} Let $\mathcal{H}$ and $\mathcal{K}$ be two infinite dimensional complex Hilbert spaces and $\phi:\mathcal{B(H)}\rightarrow  \mathcal{B(K)}$ be a surjective additive map. If $\phi$  preserves the nonzero idempotency of products of operators in both directions, then there exists a bounded linear or conjugate linear bijection $A: \mathcal{H} \rightarrow \mathcal{K}$ such that $ \phi(T)= \xi ATA^{-1}$ for every $T \in \mathcal{B(H)}$ or $ \phi(T)= \xi AT^tA^{-1}$ for every $T \in \mathcal{B(H)}$, where $ \xi = \pm 1$.
\end{theorem}
%---------------------------------------------------------------------------------------%

\section{Proofs}

In this section we prove our results. First we recall some
notations. Let $X$ and $Y$ be Banach spaces. Recall that a standard operator
algebra on $X$ is a norm closed subalgebra of $B(X)$ which
contains the identity and all finite rank operators. Denote the set of all idempotent operators of $\mathcal{B(H)}$ by $\mathcal{I(H)}$ and the Jordan product of $A,B$ by $A \circ B= \frac{1}{2}(AB+BA)$. Also denote the dual space $X$ by $ X^*$. \par
 For every nonzero $x\in X$ and
$f\in X^*$, the symbol $x\otimes f$ stands for the rank one
linear operator on $ X$ defined by
$$(x\otimes f)y=f(y)x. \hspace{.4cm} (y\in X)$$
If $x,y\in \mathcal{H}$, then
$x\otimes y$ stands for the rank one
linear operator on $ \mathcal{H}$ defined by
$$(x\otimes y)z=<z,y>x  \hspace{.4cm}(z\in\mathcal{H})$$
where $<z,y>$ denotes the inner product of $z$ and $y$.
%---------------------------------------------------------------------------------------%
We need some lemmas to prove our main result.
Let $\mathcal{A} \subseteq B(X)$ and $\mathcal{B} \subseteq B(Y)$ be
standard operator algebras.
 \par The proof of the following lemma is similar to that of Lemma 2.2 in $[13]$.
%---------------------------------------------------------------------------------------%
\begin{lemma}$[13]$
Let $\phi:\mathcal{A} \rightarrow \mathcal{B}$ be an additive map such that preserves the nonzero idempotency of Jordan products of operators. If $N\in \mathcal{A}$ is a finite rank operator such that $N^2=0$, then $\phi(N)^4=0$.
\end{lemma}
%---------------------------------------------------------------------------------------%
\begin{lemma} Let $\phi:\mathcal{A} \rightarrow \mathcal{B}$ be an additive map. Then the following statements are hold. \par
 $(i)$ If $\phi$ preserves the nonzero idempotency of Jordan products of operators, then $\phi$ is injective. \par
 $(ii)$ If $I \in \mathrm{rng} \phi$ and $\phi$ preserves the nonzero idempotency of Jordan products of operators in both directions, then $\phi (I)= I$ or $\phi (I)=-I$.
\end{lemma}
%---------------------------------------------------------------------------------------%
\begin{proof} $(i)$ Assume $ \phi (A)=0$. We assert that $A$ satisfies a
quadratic polynomial equation. Otherwise, by the discussion in $[10]$, there exists an $x \in X$ such that $x$, $Ax$ and $A^2x$ are linear independent. Then there is a linear functional $f$ such
that $f(x)=f(A^2x)=0$ and $f(Ax)=2$, because $ \dim X \geq 3$.
Setting $B=x \otimes f$, we have $A \circ B \in P_\mathcal{A}
\setminus \{0\}$, implying that
$$ \phi (A) \circ \phi (B) \in P_\mathcal{B} \setminus \{0\}.$$
This is a contradiction, because $ \phi (A) \circ \phi (B)=0$. So by the discussion in $[10]$, $A$ satisfies a
quadratic polynomial equation. \par
Assume on the contrary that $A$ is a nonzero operator. For any $B \in \mathcal{A}$ we have
 $$ \phi (A) \circ \phi (B)=0.$$
However, there exists $B=x \otimes f$ such that $A \circ B=Ax \otimes f+x \otimes fA$
is a nonzero idempotent, a contradiction. We construct such $B$. \par By the proved assertion, $A$ satisfies a quadratic polynomial equation. The spectrum of such $A$ consists only of eigenvalues. If $A^2 \neq 0$, then $A$ has a nonzero eigenvalue $\lambda$, because in this case there exist $r,s \in \mathbb{C}$ such that $rs \neq 0$ and $ \lambda $ satisfies a quadratic polynomial equation $ \alpha ^2=r \alpha +s$. Since $rs \neq 0$, $ \alpha ^2=r \alpha +s$ has a nonzero root. Let $x$ be its eigenvector. Choose a bounded functional $f$ with $f(x)= \frac{1}{2 \lambda}$  to form $B =x \otimes f$ with the desired properties. \par The remaining case is
$A^2 = 0$. Since $A$ is nonzero, we can find a vector $x$ so that $Ax \neq 0$ and a functional
$f$ with $f(x)=0$ and $f(Ax)=1$ to form $B =x \otimes f$ with the desired
properties. The proof is complete. \par
$(ii)$ Since $I \in \mathrm{rng} \phi$, there exists a
nonzero operator $U \in \mathcal{A}$ such that $ \phi (U)=I$. We show that $U=I$ or $-I$. We
have $\phi(U) \circ \phi(U)\in \mathcal{P}_{ \mathcal{B}} \setminus \{0\}$. Hence we obtain
$$U^2=U \circ U \in \mathcal{P}_{ \mathcal{A}} \setminus \{0\}. \leqno(1)$$
This implies that for any $x \in X$, $x$, $Ux$ and $U^2x$ are linear dependent. Thus $U$ satisfies a quadratic polynomial equation, by the discussion in $[10]$. This together with $(1)$ yields that there exist $a,b \in  \mathbb{C} $ such that we have
$$U^2=aU+bI. \leqno(2)$$
From $(1)$ and $(2)$, we obtain the answers $(0,1)$, $(1,0)$ and $(-1,0)$ for $(a,b)$ which imply that $U^2=I$, $U^2=U$ýand $U^2=-U$. \par
Let $U^2=U$. We assert that $U=I$. Assume on the contrary that $U \neq I$. From $ U \circ I \in \mathcal{P}_\mathcal{A} \setminus \{0\}$, we obtain
$$ \phi (I)=I \circ \phi (I)  \in \mathcal{P}_\mathcal{B} \setminus \{0\}. \leqno(3)$$
 On the other hand, there exists an idempotent operator $T$ such that
$U-T$ is not idempotent. In fact, $U+S-I$ isn't idempotent, where
$S=I-T$. Thus
$$(U+S-I) \circ I=U+S-I \not \in
\mathcal{P}_\mathcal{A} \setminus \{0\}$$
 which implies that
 $$ \phi (U+S-I) \circ \phi ( I) \not \in \mathcal{P}_\mathcal{B} \setminus \{0\}$$
This together with $(3)$ yields that
$$ \phi (I) \circ \phi(S) \not \in \mathcal{P}_\mathcal{B} \setminus \{0\}$$
which implies that
$$S=I \circ S \not \in \mathcal{P}_\mathcal{A} \setminus \{0\}.$$
This is a contradiction, because $S$
is idempotent. So the proof of assertion is completed. \par
With a similar proof, the assumption $U^2=-U$ yields that $U=-I$. \par
Now let $U^2=I$. We assert that $U$ is a multiple of $I$. Assume on the contrary that $U$ is
a non-scalar operator. Since $I$ and $U$ are linear independent,
there is a nonzero vector $x \in X$ such that $x$ and $Ux$ are
linear independent. Hence there exists $f \in X^*$ such that
$f(x)=0$ and $f(Ux)=1$. Setting $B=x \otimes f$, we obtain
$$U \circ B \in \mathcal{P}_\mathcal{A} \setminus \{0\}$$
which implies that
$$ \phi (B)= \phi (U) \circ \phi (B) \in \mathcal{P}_\mathcal{B} \setminus \{0\}.$$
This is a contradiction, because $B$ is a nilpotent such that $B^2=0$ and so by Lemma 2.1, $\phi (B)$ is a
nilpotent operator. So the proof of assertion is completed.
By the proved assertion, there exists a nonzero
complex number $\lambda$ such that $U= \lambda I$. Since $U^2=I$, we obtain $ \lambda
^2=1$ and this completes the proof.
\end{proof}
%---------------------------------------------------------------------------------------%
\begin{theorem} $[4]$
Let $\mathcal{H}$ and $\mathcal{K}$ be two infinite dimensional real or complex Hilbert spaces and $\phi:\mathcal{B(H)}\rightarrow  \mathcal{B(K)}$ be an additive map preserving idempotents. Suppose that the range of $\phi$ contains all minimal idempotents. Then $\phi$ either annihilates minimal idempotents or there exists a bounded linear or conjugate linear bijection $A: \mathcal{H} \rightarrow \mathcal{K}$ such that $ \phi(T)= ATA^{-1}$ for every $T \in \mathcal{B(H)}$ or $ \phi(T)= AT^tA^{-1}$ for every $T \in \mathcal{B(H)}$ ( in the case that $\mathcal{H}$ and $\mathcal{K}$ are real, $A$ is linear).
\end{theorem}
%---------------------------------------------------------------------------------------%
\par \vspace{.3cm}
\textbf{Proof of Theorem 1.1.} Since by Lemma 2.2, $\phi (I)= I$ or $\phi (I)= -I$, from $P=I \circ P \in \mathcal{I(H)} \setminus \{0\}$ we obtain that $ \phi (P)$ or $- \phi (P)$ belongs to $ \mathcal{I(H)} \setminus \{0\}$. This together with $\phi (0)=0$ implies that $\phi$ or $- \phi$ preserves the idempotent operators in both directions. Hence the forms of $\phi$ follows from Theorem 2.3.
%---------------------------------------------------------------------------------------%
\begin{proposition} Let $ \dim \mathcal{H} \geq 3$. Let $A$ be an arbitrary operator of $\mathcal{B(H)}$ and $P$ be a rank one idempotent operator. Then $A \in \mathbb{C}^*P$ if and only if for every $T \in \mathcal{B(H)}$ such that $PT \in \mathcal{I(H)} \setminus \{0\}$ we have $AT \not \in \mathcal{I(H)} \setminus \{0\}$, where $ \mathbb{C}^*= \mathbb{C} \setminus \{0,1\}$.
\end{proposition}
%---------------------------------------------------------------------------------------%
\begin{proof} If $A \in \mathbb{C}^*P$, then there exists a $ \lambda \in \mathbb{C}^*$ such that $A= \lambda P$. hence it is trivial that for every $T$ such that $PT \in \mathcal{I(H)} \setminus \{0\}$ then $ \lambda PT \not \in \mathcal{I(H)} \setminus \{0\}$. \par
Conversely, Let $A \not \in \mathbb{C}^*P$. Since $P$ is rank one, by $[2]$, there exists either an $x \in \mathcal{H}$ such that $Px$ and $Ax$ are linear independent or an $x \in \mathcal{H}$ and linear independent vectors $z_1,z_2 \in \mathcal{H}$ such that $P=x \otimes z_1$ and $A=x \otimes z_2$. \par If $Px$ and $Ax$ are linear independent, then there exists $y \in \mathcal{H}$ such that $<Px,y>=<Ax,y>=1$, because $ \dim \mathcal{H} \geq 3$. Setting $T=x \otimes y$ follows that $PT \in \mathcal{I(H)} \setminus \{0\}$ and also $AT \in \mathcal{I(H)} \setminus \{0\}$. This is a contradiction. \par
If $P=x \otimes z_1$ and $A=x \otimes z_2$, then there exist $y,z_3 \in \mathcal{H}$ such that $<x,z_3>=<y,z_1>=<y,z_2>=1$. Setting $T=y \otimes z_3$ follows that $PT \in \mathcal{I(H)} \setminus \{0\}$ and also $AT \in \mathcal{I(H)} \setminus \{0\}$. This is a contradiction. \par
These contradictions yields that $A \in \mathbb{C}^*P$ and this completes the proof.
\end{proof}
%---------------------------------------------------------------------------------------%
\begin{lemma} Let $\phi:\mathcal{A} \rightarrow \mathcal{B}$ be a surjective additive map such that
 $$AB \in \mathcal{P}_\mathcal{A} \setminus \{0\} \Leftrightarrow \phi (A) \phi (B) \in \mathcal{P}_\mathcal{B} \setminus \{0\}$$
 for every $A \in \mathcal{A}$ and $B \in \mathcal{B}$. Then the following statements are hold. \par
  $(i)$  $\phi (I)= I$ or $\phi (I)=-I$. \par
  $(ii)$ If $ \mathcal{A} = \mathcal{B(H)}$ with $ \dim \mathcal{H} \geq 3$ and $ \mathcal{B}= \mathcal{B(K)}$, then $ \phi ( \mathbb{C}P ) \subseteq \mathbb{C} \phi (P)$, for every rank one idempotent $P$.
\end{lemma}
%---------------------------------------------------------------------------------------%
\begin{proof} $(i)$ It is proved by using Lemma 2.1 and similar to the proof of Lemma 2.2 in $[13]$. \par
  $(ii)$ By $(i)$, $\phi (I)= I$ or $\phi (I)=-I$. Since $\phi (0)=0$, we can conclude from $(i)$ that $\phi$ or $- \phi$ preserves the idempotent operators in both directions. By Lemma 2.6 in $[12]$, $\phi$ or $- \phi$ preserves the rank one idempotent operators in both directions. If $A \in \mathbb{C}^*P$, then by Proposition 2.4, for every $T \in \mathcal{B(H)}$ such that $PT \in \mathcal{I(H)} \setminus \{0\}$ we have $AT \not \in \mathcal{I(H)} \setminus \{0\}$ which by surjectivity of $\phi$ imply that for every $T' \in \mathcal{B(H)}$ such that $ \phi (P)T' \in \mathcal{I(H)} \setminus \{0\}$ we have $ \phi (A)T' \not \in \mathcal{I(H)} \setminus \{0\}$. Since $ \phi (P)$ is a rank one idempotent, by Proposition 2.4 we can conclude that $ \phi (A) \in \mathbb{C}^* \phi (P)$. This together with $(i)$ and $\phi (0)=0$ follows that $ \phi ( \mathbb{C}P ) \subseteq \mathbb{C} \phi (P)$. This completes the proof.
\end{proof}
%---------------------------------------------------------------------------------------%
\begin{proposition} Let $\mathcal{H}$ and $\mathcal{K}$ be two infinite dimensional real or complex Hilbert spaces and $\phi:\mathcal{B(H)}\rightarrow  \mathcal{B(K)}$ be an additive map. If $\phi$ preserves the idempotent operators, then $\phi$ preserves the square zero operators.
\end{proposition}
%---------------------------------------------------------------------------------------%
\begin{proof} Let $N \in \mathcal{B(H)}$ be a square zero operator. Then we have
$ \mathcal{H}= \ker N \oplus M$ for some closed subspace $M$ of $\mathcal{H}$. Thus by
this decomposition $N$ has the following operator matrix
 $$ N=\begin{pmatrix}
     0 & N_1 \\
     0 & 0 \\
   \end{pmatrix}.$$
 If  $$ A=\begin{pmatrix}
     I & 0 \\
     0 & 0 \\
   \end{pmatrix}$$
then $A+nN\in \mathcal{I(H)}$ for every natural number $n$. It implies that $ \phi (A)+n \phi (N) \in \mathcal{I(K)}$ for every natural number $n$. That is,
$$\phi(A)+n\phi(N)=\phi(A)^2+n(\phi(A)\phi(N)+\phi(N)\phi(A))+n^2\phi(N)^2$$
for all $n$. Setting $n=1$ and $n=2$
yield
$$\phi(N)=\phi(A)\phi(N)+\phi(N)\phi(A)+\phi(N)^2,$$
$$2\phi(N)=2(\phi(A)\phi(N)+\phi(N)\phi(A))+4\phi(N)^2$$
which imply that $\phi(N)^2=0$ and this completes the proof.
\end{proof}
%---------------------------------------------------------------------------------------%
\begin{theorem} $[1]$ Let $\mathcal{H}$ and $\mathcal{K}$ be two infinite dimensional complex Hilbert spaces and $\phi:\mathcal{B(H)}\rightarrow  \mathcal{B(K)}$ be a surjective additive map such that $ \phi ( \mathbb{C}P ) \subseteq \mathbb{C} \phi (P)$ holds for every rank one operator $P$. Then $\phi$ preserves square zero in both
directions if and only if there exists a nonzero scalar c and a bounded linear or conjugate linear bijection $A: \mathcal{H} \rightarrow \mathcal{K}$ such that $ \phi(T)= c ATA^{-1}$ for every $T \in \mathcal{B(H)}$ or $ \phi(T)= c AT^tA^{-1}$ for every $T \in \mathcal{B(H)}$.
\end{theorem}
%---------------------------------------------------------------------------------------%
\par \vspace{.3cm}
\textbf{Proof of Theorem 1.2.} By a similar proof to that of Lemma 2.3 in $[13]$, we obtain that $\phi$ is injective.
Since $\phi (0)=0$, we can conclude from part $(i)$ of Lemma 2.5 that $\phi$ or $- \phi$ preserves the idempotent operators in both directions. This together with Proposition 2.6 and the injectivity of $\phi$ implies that $\phi$ preserves the square zero operators in both directions. Moreover, by part $(i)$ of Lemma 2.5, $ \phi ( \mathbb{C}P ) \subseteq \mathbb{C} \phi (P)$, for every rank one idempotent $P$. Therefore the forms of $\phi$ follow from Theorem 2.7. The scalar c in Theorem 2.7 is the scalar that $\phi (I)= cI$ (by the proof of this theorem in $[1]$).
This together with the part $(i)$ of Lemma 2.5 completes the proof.
%---------------------------------------------------------------------------------------%
\par \vspace{.4cm}{\bf Acknowledgements:} This research is partially
supported by the Research Center in Algebraic Hyperstructures and
Fuzzy Mathematics, University of Mazandaran, Babolsar, Iran.

%---------------------------------------------------------------------------------------%
\bibliographystyle{amsplain}

\end{document}